\documentclass[amscd, amssymb, verbatim, 12pt]{amsart}
\usepackage{amssymb, stmaryrd, graphicx, caption2}

\def\be{\begin{eqnarray}}
\def\ee{\end{eqnarray}}
\def\bea{\begin{eqnarray*}}
\def\eea{\end{eqnarray*}}
\def\na{\nabla}

\def\vp{\varphi}
\def\r{\rho}
\def\a{\alpha}
\def\n{\nabla}

\def\o{\omega}

\newtheorem{defi}{Definition}[section]

\newtheorem{lem}[defi]{Lemma}
\newtheorem{remark}[defi]{Remark}
\newtheorem{thm}[defi]{Theorem}
\newtheorem{prop}[defi]{Proposition}
\newtheorem{cor}[defi]{Corollary}

\numberwithin{equation}{section}

\begin{document}

\title[Ridigity of Ricci Solitons with Weakly Harmonic Weyl Tensors]
{ Ridigity of Ricci Solitons with Weakly Harmonic Weyl Tensors}

\author[S. Hwang]{Seungsu Hwang}
\thanks{The first author was supported by the Ministry of Education (NRF-2015R1D1A1A01057661).}

\address{Department of Mathematics\\
Chung-Ang University\\ 221 HeukSuk-dong, DongJak-ku\\ Seoul, Korea 156-756}
\email{seungsu@cau.ac.kr}

\author[G. Yun]{Gabjin Yun$^*$}
\thanks{$^*$ The second author is the corresponding author.
This work was supported by the Basic Science Research Program through the National Research Foundation
of Korea (NRF) funded by the Ministry of Education, Science and Technology (Grant No. 2011-0007465).}

\address{Department of Mathematics\\
Myong Ji University\\
San 38-2, Nam-dong, Yongin-si\\
Gyeonggi-do, Korea, 449-728}
\email{gabjin@mju.ac.kr}


\begin{abstract}
In this paper, we prove rigidity results on gradient shrinking Ricci solitons with weakly harmonic Weyl curvature tensors. Let $(M^n, g)$ be a compact gradient shrinking Ricci soliton satisfying
${\rm Ric}_g + Ddf = \rho g$ with $\rho >0$ constant. We show that if $(M,g)$ satisfies
$\delta \mathcal W (\cdot, \cdot, \nabla f) = 0$, then $(M, g)$ is Einstein. 
Here $\mathcal W$ denotes the Weyl curvature tensor. In the case of noncompact, if $M$ is complete and satisfies the same condition, then $M$ is rigid in the sense that $M$ is given by a quotient of product of an Einstein manifold with Euclidean space. These are generalizations of the previous known results in
 \cite{l-r}, \cite{m-s} and \cite{p-w3}.
 \end{abstract}

\maketitle


\section{Introduction}

\setlength{\baselineskip}{22pt}

A complete Riemannian metric $g$ on a smooth manifold $M^n$ is called a {\it Ricci soliton} if there exist a constant
$\r$ and a smooth $1$-form $\o$ such that
\be
2r_g + {\mathcal L}_{\o^\sharp} g = 2 \rho g,\label{eqn1-1-1}
\ee
where $r_g$ is the Ricci tensor of the metric $g$, $\o^\sharp$ is the vector field  that is dual to $\o$,
 and ${\mathcal L}_{\o^\sharp}$ denotes the Lie derivative along $\o^\sharp$.
Since
$$
\mathcal L_{\o^\sharp} g(X, Y) = D_X \o(Y) + D_Y\o(X)$$
 for any vector fields $X$ and $Y$,
$(\ref{eqn1-1-1})$ is equivalent to
\be
2r_g(X, Y) + D_X \o(Y) + D_Y\o(X) = 2 \rho g(X, Y).\label{eqn1-1-2}
\ee
Moreover, if there exists a smooth function $f$ on $M$ such that $\o = df$,
then $g$ is called a {\it gradient Ricci soliton}. The Ricci soliton
is said to be {\it shrinking, steady} or {\it expanding} according as $\rho >0, \rho = 0$, or $\rho <0$, respectively.
In the case of a gradient Ricci soliton, (\ref{eqn1-1-2}) becomes
\be
r_g + Ddf = \rho g. \label{eqn1-1-3}
\ee
Clearly, Einstein metrics are Ricci solitons with $f$ being trivial.
Another interesting special case occurs when $f(x) = \frac{\rho}{2}|x|^2$ on ${\Bbb R}^n$. 
In this case,
$$
Ddf = \rho g,
$$
which gives  a gradient Ricci soliton where the background metric is flat.
This example is called a Gaussian.

Generalizing the trivial Ricci solitons, Petersen and Wylie (\cite{p-w}) introduced the notion of rigidity for
gradient Ricci solitons. A gradient Ricci soliton is said to be {\it rigid} if it is isometric to a quotient
of $N \times {\Bbb R}^k$, where $N$ is an Einstein manifold and $f = \frac{\rho}{2}|x|^2$ on the Euclidean factor.
That is, the Riemannian manifold $(M, g)$ is isometric to $N\times_{\Gamma} {\Bbb R}^k$, where $\Gamma$ acts freely on $N$ and by
orthogonal transformations on ${\Bbb R}^k$. When $M$ is compact, a Ricci soliton $(M, g)$ is rigid if and only if it is
Einstein. The rigidity of gradient Ricci solitons has been studied, for example, in \cite{l-r, na, p-w, p-w2}.
Note that Einstein manifolds have harmonic Weyl tensor,
and Ricci solitons can be considered as generalizations of Einstein metrics.
Therefore, it is natural to ask about the relation between rigidity  and  harmonicity of the Weyl tensor on a Ricci soliton.
In this regard, Cao-Wang-Zhang (\cite{c-w-z}), Ni-Wallach (\cite{n-w}), and Petersen-Wylie (\cite{p-w3})
have studied the classification of complete gradient shrinking Ricci solitons
with vanishing Weyl curvature tensor under certain assumptions on the Ricci curvature.
In \cite{l-r}, Fern\'andez-Lop\'ez and Gar\'ia-R\'io proved that if $(M, g)$ is a compact Ricci soliton, then
$(M, g)$ is rigid if and only if it has harmonic Weyl tensor. We say that a Riemannian manifold $(M, g)$ has harmonc Weyl tensor if
$\delta \mathcal W = 0$, or equivalently if $r_g - \frac{s_g}{2(n-1)}g$ is a Codazzi tensor,
where $\mathcal W$ and  $s_g$ denote the Weyl curvature tensor and  the
scalar curvature of the metric $g$, respectively, and $\delta$ is the divergence operator, which is the
adjoint of the differential operator $d$ acting on tensors.
 In \cite{m-s}, Munteanu and Sesum extended these results to show that
any $n$-dimensional complete gradient Ricci soliton with harmonic Weyl tensor is a finite quotient of ${\Bbb R}^n,\,
{\Bbb S}^{n-1}\times {\Bbb R},\, $ or ${\Bbb S}^n$.

In this paper, we consider gradient shrinking Ricci solitons with weakly harmonic Weyl tensor.
A gradient Ricci soliton $(M, g, f)$ is said to have weakly harmonic Weyl tensor if
\be
\delta \mathcal W (\cdot, \cdot, \nabla f) = 0.\label{eqn3-25-6}
\ee
Note that this class includes gradient Ricci solitons $(M, g)$ that have locally conformally flat metrics
or more generally harmonic Weyl tensors.

Some results exist  concerning  gradient Ricci solitons and quasi-Einstein manifolds $(M, g, f)$ for $n \ge 3$,
with a relaxed Weyl curvature condition 
rather than a local conformal flatness  or a harmonic Weyl curvature condition.
For example, G. Catino (\cite{cat}) studied generalized quasi-Einstein manifolds with harmonic Weyl tensor
such that ${\mathcal W}(\n f, \cdot, \cdot, \cdot) = 0$. In particular, he proved that  the condition
${\mathcal W}(\n f, \cdot, \cdot, \cdot) = 0$ for a quasi-Einstein manifold implies that
the conformal metric $\tilde g = e^{-\frac{2}{n-2}f}g$ has harmonic Wely tensor.
P. Petersen and W. Wylie (\cite{p-w2}) proved
a classification theorem on complete gradient Ricci solitons $(M^n, g, f)$ for $n \ge 3$
 with constant scalar curvature, and ${\mathcal W}(\n f, \cdot, \cdot, \n f) = {\rm o} \left(|\n f|^2\right)$.


In this paper, we prove that if  $(M, g)$ is a compact gradient shrinking Ricci soliton having weakly harmonic Weyl
curvature tensor, then $(M, g)$ is Einstein.

\vspace{.2in}
\noindent
{\bf Theorem A}\,\,
Let $(M^n, g, f)$ be a {compact} gradient shrinking Ricci soliton, 
and uppose that $\delta \mathcal W(\cdot, \cdot, \nabla f) = 0$.
Then $(M, g)$ is Einstein.
\vspace{.2in}

In the case of a noncompact Ricci soliton, we  prove the following rigidity reult which is exactly the same property
as  when $(M, g)$  is locally conformally flat or has  harmonic Weyl curvature.

\vspace{.2in}
\noindent
{\bf Theorem B}\,\,
Let $(M^n, g, f)$ be a {complete noncompact} gradient shrinking Ricci soliton,
suppose that $\delta \mathcal W(\cdot, \cdot, \nabla f) = 0$.
Then $(M, g)$ is rigid.
\vspace{.2in}

One of main ingredients in proving the rigidity of Ricci solitons with vanishing Weyl curvature tenosr
is that the condition $\delta \mathcal W$  makes it possible to compute the divergence 
of the full Riemannian curvature tensor $R$. 
Therefore, by  using an integral identity containing the divergence of $R$ as an integrand, one can show that
the scalar curvature must be constant. 

However, in the case of a weakly harmonic Weyl curvature tensor, it is not easy 
to deduce the divergence of $R$.  Thus, we must find an alternative  approach to prove our rigidity result for 
 Ricci solitons with weakly harmonic Weyl curvature tensor.
The key observation concerning  a gradient shrinking Ricci soliton satisfying 
$\delta \mathcal W (\cdot, \cdot, \nabla f) = 0$ is the following.
If $(M, g, f)$ is a  gradient Ricci soliton  having weakly harmonic Weyl tensor,
then two gradient vector fields $\n s_g$ and $\n f$ are parallel, and this property implies that 
$s_g, |\n f|^2$, and $r_g(\n f, \n f)$ are all constant along each level hypersurface given by $f$.
Using these properties, together with a maximum principle and a Liouville property
 for $f$-Laplacian given by  Petersen and Wylie (\cite{p-w}), 
we can derive our main results.

\section{Preliminaries and Basic Formulas}

Throughout this paper, we follow the conventions in \cite{besse} regarding differential operators and the Riemannian curvature tensor $R$, with the  exception of only  the Laplacian.
We define the Laplacian by $\Delta = -(\delta d + d \delta)$, the negative operator. 
For example, $\displaystyle{\Delta \vp = \frac{\partial^2 \vp}{\partial x^2} 
+\frac{\partial^2 \vp}{\partial y^2}}$  for a function
$\vp : {\Bbb R}^2 \to {\Bbb R}$.

We start with basic definitions of differential operators acting on tensors.
Let us denote by $C^{\infty}(S^2M)$ the space of sections of symmetric $2$-tensors on a Riemannian manifold $M$. 
Let $D$ be the Levi-Civita connection of $(M, g)$. Then
the differential operator $d^D$ from $C^{\infty}(S^2M)$ into  $C^\infty\left(\Lambda^2 M \otimes T^*M\right)$ is defined as
$$ 
d^D \eta(X,Y,Z)= (D_X \eta)(Y,Z)-(D_Y \eta)(X,Z)
$$
for $\eta \in C^{\infty}(S^2M)$ and vectors $X, Y$, and $Z$. 
 For a function $\varphi \in C^\infty(M)$ and $\eta \in C^{\infty}(S^2M)$,
$d\varphi \wedge \eta $ is defined as
$$
(d\varphi \wedge \eta) (X,Y,Z)= d\varphi(X) \eta(Y,Z)-d\varphi(Y) \eta(X,Z).
$$
Here, $d\varphi$ denotes the usual total differential of $\varphi$.
Then, it clearly follows that
\be
d^D(\varphi\eta) = d\varphi \wedge \eta + \varphi d^D \eta.\label{eqn5}
\ee
We also define two types of interior product $\iota$ and $\hat{\iota}$ for the curvature tensor $R$, by
$$
\iota_{_V}R(X, Y, Z) = R(V, X, Y, Z), \quad \hat{\iota}_{_V} R(X, Y, Z) = R(X, Y, Z, V)
$$
for vectors $V, X, Y,$ and $Z$.

Next, we will describe some basic formulas that can easily be obtained from the Ricci soliton equation (\ref{eqn1-1-3}) which are already well known.
From (\ref{eqn1-1-3}), we have
\be
\Delta f = n \rho - s_g,\quad d\Delta f = - ds_g, \quad \Delta s_g = - \Delta^2 f.\label{eqn1-1-4}
\ee
Note that for any smooth function $\varphi$, we have
\be
\delta Dd\varphi = - d\Delta \varphi - r_g(\n \varphi, \cdot).\label{eqn1-1-5}
\ee
Taking the divergence of (\ref{eqn1-1-3}) and using (\ref{eqn1-1-5}), we have
\be
- \frac{1}{2}ds_g - d\Delta f - r_g(\na f, \cdot) = 0.\label{eqn2016-2-6-1}
\ee
By (\ref{eqn1-1-4}) and (\ref{eqn2016-2-6-1}), it holds that
\bea
\iota_{_{\na f}}r_g = r_g(\na f, \cdot) = \frac{1}{2}ds_g. \label{eqn1-1-6}
\eea
Furthermore, we can see from (\ref{eqn2016-2-6-1}) together with the Ricci soliton equation (\ref{eqn1-1-3}) that
\bea
s_g+ |df|^2 - 2 \rho f = C\,\, {\rm  (constant).} \label{eqn2-6-1}
\eea
Since $\displaystyle{\delta r_g = - \frac{1}{2}ds_g}$, it holds that
\bea
\frac{1}{2}\delta ds_g
&=&
-\frac{1}{2}\langle \na s_g, \na f\rangle - \langle r_g, Ddf\rangle.
\eea
Therefore,
\be
\Delta s_g = \langle \na s_g, \na f\rangle + 2\langle r_g, Ddf\rangle.\label{eqn1-1-7}
\ee
It follows from (\ref{eqn1-1-3}) again that
\be
\langle r_g, Ddf\rangle = \r \Delta f - |Ddf|^2\label{eqn1-5-1-100}
\ee
and
\be
\langle r_g, Ddf\rangle = \r s_g - |r_g|^2.\label{eqn1-5-1}
\ee
Substituting (\ref{eqn1-5-1})  into (\ref{eqn1-1-7}), we obtain
\be
\Delta s_g = \langle \na s_g, \na f\rangle +  2\r s_g - 2|r_g|^2.\label{eqn1-5-2}
\ee
This  implies that  a gradient shrinking Ricci solition  has nonnegative scalar curvature.
In fact, if $\displaystyle{\min_M s_g = s_g(x_0)}$, then $\rho s_g(x_0) \ge |r_g|^2(x_0) \ge 0$.
Of course, this fact is well known (\cite{chen}).
Moreover, substituting (\ref{eqn1-5-2}) into (\ref{eqn1-1-7}) and using (\ref{eqn1-5-1-100}), we get
\bea
\Delta (2\r f - s_g) + \langle \na s_g, \na f\rangle = 2|Ddf|^2.\label{eqn1-1-8}
\eea

Next, we state some well-known facts that are needed to prove our main theorems.

\begin{lem}[\cite{besse}]\label{lem2011-12-26-10}
For any Riemannian manifold $(M, g)$,
 we have
\bea
\delta {\mathcal W} = - \frac{n-3}{n-2}d^D \left( r_g - \frac{s_g}{2(n-1)}g\right)\label{eqn10}
\eea
 under the  identification of $T^*M \otimes \Lambda^2M$ with  $\Lambda^2M \otimes T^*M$.
\end{lem}

\begin{lem}\label{lem2011-12-26-1-1}
For any function $f$ on a Riemannian manifold $(M, g)$, it holds that
\bea
d^D (Ddf)(X, Y, Z) =  \langle R(X, Y)Z, \n f\rangle = (\hat{\iota}_{_{\n f}}R)(X, Y, Z).\label{eqn2}
\eea
\end{lem}

Before closing this section, we will prove a result concerning the   rigidity of 
 complete gradient shrinking Ricci solitons with constant scalar curvature and weakly harmonic Weyl
tensor.  In the compact case, this property is already known. In fact, a compact gradient shrinking Ricci
 soliton with constant scalar curvature is known to be Einstein without requiring any condition on the Weyl tensor.
To prove our rigidity, we will require  the following theorem given by  Petersen and Wylie.

\begin{thm}[\cite{p-w}]\label{thm2-6-2}
A shrinking gradient Ricci soliton is rigid if and only if it has constant scalar curvature and is radially flat.
\end{thm}
We say that a gradient Ricci soliton $(M, g, f)$ is radially flat if the sectional curvature of the plane spanned by 
$\nabla f$ and an orthogonal vector to $\nabla f$ vanishes.

\begin{thm}\label{thm2013-12-9-1}
Let $(M^n, g, f)$ be a complete gradient shrinking Ricci soliton, and suppose that 
$\delta \mathcal W(\cdot, \cdot, \nabla f) = 0$.
If $s_g$ is constant, then $(M, g)$ is rigid.
\end{thm}
\begin{proof}
If $M$ is compact, then $(M, g)$ is Einstein by (\ref{eqn1-1-4}). 
Therefore, we may assume that  $(M,g)$ is complete and noncompact.
Since the scalar curvature $s_g$ is constant,  it follows from  Lemma~\ref{lem2011-12-26-10} that 
\bea
\delta {\mathcal W}  = - \frac{n-3}{n-2} \, d^D  r_g.
\eea
Thus,
\be
d^D r_g(X, \n f, Y) = - \frac{n-2}{n-3}\delta  {\mathcal W} (Y, X, \n f) = 0\label{eqn2013-9-29-1}
\ee
for any vectors $X$ and $Y$. From the soliton equation (\ref{eqn1-1-3}) 
 together with Lemma~\ref{lem2011-12-26-1-1}, we have
\be
d^D r_g = - d^D Ddf = \hat{\iota}_{_{\n f}} R, \label{eqn6-1}
\ee
and so for any vectors $X$ and $Y$, it holds that
\be
d^D r_g(X, Y, \n f) = 0.\label{eqn10-2-2}
\ee
The equations (\ref{eqn2013-9-29-1})-(\ref{eqn10-2-2}) show that
$$
0 = d^D r_g(\n f, X, Y)  = R(\n f, X, Y, \n f).
$$
Thus,  $(M, g)$ is rigid by Theorem~\ref{thm2-6-2}.
\end{proof}

\section{Ricci solitons with weakly harmonic Weyl tensor}

In this section, we will describe  some properties for a gradient Ricci soliton with weakly harmonic Weyl tensor.
As mentioned above, one key property of such a gradient Ricci soliton is that gradient vector fields
$\nabla f$ and $\nabla s_g$ are parallel, which implies that all functions,
 including the scalar curvature $s_g$ and
$|\n f|^2$, are constants along each level set of $f$. The second result regarding such 
gradient Ricci solitons is that
the Ricci tensor can be decomposed into the direction $\n f$ and its orthogonal direction, which
implies that the Ricci tensor has at most two eigenvalues of multiplicity $n-1$ and  $1$, respectively.
 This property plays an important role in investigating the structure of harmonic  curvature (cf. \cite{der}, \cite{laf}).

When the scalar curvature $s_g$ is constant, we already know  that $(M, g)$ is rigid.
Therefore, in this section, we carry out  various computations on
gradient Ricci solitons with  the assumption that the scalar curvature $s_g$ is nonconstant.

\begin{lem}
Let $(M^n, g, f)$ be a gradient Ricci soliton with weakly harmonic Weyl tensor.
Then, for any vectors $X$ and $Y$,
\be
D_Y r_g (\n f, X) = D_{\n f} r_g (X, Y) + \frac{1}{2(n-1)} \left(df \otimes ds_g - \langle \n f, \n s_g\rangle g \right)(X, Y).\label{eqn20-5-1}
\ee
\end{lem}
\begin{proof}
It follows from  Lemma~\ref{lem2011-12-26-10}, together with 
the assumption  $\delta \mathcal W(\cdot, \cdot, \n f) = 0$, that
\bea
d^D r_g(Y, \n f, X) = \frac{1}{2(n-1)}ds_g\wedge g(Y, \n f, X) \label{eqn15-5}
\eea
for any vectors $X$ and $Y$. This is equivalent to (\ref{eqn20-5-1}).
\end{proof}

\begin{lem}\label{lem2013-12-10-1}
Let $(M^n, g, f)$ be a gradient Ricci soliton with weakly harmonic Weyl tensor.
Then, $\n f$ and $\n s_g$ are  parallel.
\end{lem}
\begin{proof}
Since
$$
d^D r_g(X, Y, \n f) + d^D r_g(Y, \n f, X) + d^D r_g(\n f, X, Y) =0
$$
and $d^D r_g(X, Y, \n f) = 0$, we have
\be
D_Xr_g(Y, \n f) = D_Yr_g(X, \n f)\label{eqn2013-9-13-7}
\ee
for any vectors $X$ and $Y$.
Switching the roles of $X$ and $Y$ in (\ref{eqn20-5-1}), we get
\bea
D_Xr_g(\n f, Y) = D_{\n f}r_g(X, Y) + \frac{1}{2(n-1)} \left(df \otimes ds_g - \langle \n f, \n s_g\rangle g \right)(Y, X).\label{eqn20-5-10}
\eea
Comparing this with (\ref{eqn20-5-1}), we obtain
$$
\langle \n s_g, X \rangle \langle \n f, Y\rangle = \langle \n s_g, Y \rangle \langle \n f, X\rangle
$$
for any vectors $X$ and $Y$. Thus,
\bea
\langle\n s_g, X \rangle  \n f = \langle \n f, X \rangle \n s_g\label{eqn2013-11-1-1}
\eea
for any vector $X$.
\end{proof}

Denote   the set of all critical points of 
$f$ and $s_g$  by ${\rm Crit}(f)$ and ${\rm Crit}(s_g)$, respectively. Then, the identity
$ds_g = 2 r_g(\n f, \cdot)$ implies that
$$
{\rm Crit}(f) \subset {\rm Crit}(s_g).
$$
Furthermore, by Lemma~\ref{lem2013-12-10-1}  we have
\bea
\n s_g =  2 r_g(N, N) \n f\label{eqn2013-12-26-1}
\eea
 on the set $M \setminus {\rm Crit}(f)$. Here, $N = \frac{\n f}{|\n f|}$.
Let
$$
\a:=r_g(N, N),
$$
so that
\be
\n s_g = 2\a \n f\label{eqn2015-6-14-2-1}
\ee
on the set $M \setminus {\rm Crit}(f)$.
Note that even though $\a$ is defined on the set $M \setminus {\rm Crit}(f)$, 
 $\a$ can be extended as a $C^0$ function on the whole of $M$ since
 $|\a| \le |r_g|$.

 \begin{lem}\label{lem2015-7-14-10}
Let $(M^n, g, f)$ be a complete gradient shrinking Ricci soliton,
and assume that $\delta \mathcal W(\cdot, \cdot, \nabla f) = 0$.
Then, for a vector field $X$ orthogonal to $\n f$ it holds that
$$
r_g(X, \n f) = 0.
$$
Furthermore, for unit vector fields $X$ and $Y$ that are orthogonal to $\n f$ with $X \perp Y$, it holds that 
$$
R(X, \n f, Y, \n f) = 0
$$
and
$$
R(X, \n f, X, \n f) = \frac{1}{n-1}r_g(\n f, \n f) 
= \frac{\a |\n f|^2}{n-1}.
$$
\end{lem}
\begin{proof}
For a vector field $X$ that is orthogonal to $\n f$, we have $\langle \n s_g, X\rangle = 0$,
 since $\n s_g$ and $\n f$ are parallel. Therefore,
$$
r_g(X, \n f) = \frac{1}{2} ds_g(X) = 0.
$$
Next, by Lemma~\ref{lem2011-12-26-1-1} and (\ref{eqn15-5}) we have
\bea
R(X, \n f, Y, \n f) &=&
d^D r_g(X, \n f, Y)\\
&=& \frac{1}{2(n-1)}ds_g\wedge g(X, \n f, Y) = 0,
\eea
and
\bea
R(X, \n f, X, \n f) &=&
d^D r_g(X, \n f, X) \\
&=& \frac{1}{2(n-1)}ds_g\wedge g(X, \n f, X)\\
&=&
\frac{1}{2(n-1)}ds_g(\n f) 
= \frac{\a |\n f|^2}{n-1}.
\eea
\end{proof}

\begin{remark}
{\rm 
As an application of Lemma~\ref{lem2015-7-14-10}, we can show that the Ricci tensor $r_g$  has
at most two eigenvalues of multiplicitiy $1$ and $n-1$. In fact, from the curvature decomposition (cf. \cite{besse}) and
Lemma~\ref{lem2011-12-26-10}, we can compute
\be
&&- \frac{n-2}{n-3}\delta {\mathcal W} + \frac{1}{2(n-1)}ds_g \wedge g 
- \frac{1}{n-2} \iota_{\n f}r_g \wedge g \label{eqn3-25-1}\\
&&\qquad + \frac{s_g}{(n-1)(n-2)}df \wedge g - \frac{1}{n-2}df \wedge r_g = 0.\nonumber
\ee
Let $\{e_1, \cdots, e_{n-1}, N\}$ be a local frame. 
Substituting the triple $(N, e_i, e_i)$ into (\ref{eqn3-25-1}), we obtain
$$
- \frac{1}{|\n f|} \frac{n-2}{n-3}\delta{\mathcal W}(N, e_i, e_i) + \frac{s_g - \a}{(n-1)(n-2)}
- \frac{1}{n-2}r_g(e_i, e_i) = 0.
$$
By considering Lemma~\ref{lem2011-12-26-10} again, since
$$ - \frac{n-2}{n-3}\delta {\mathcal W}(N, e_i, e_i) = d^D r_g(e_i, e_, N) - 
\frac{1}{2(n-1)}ds_g \wedge g (N, e_i, e_i) = 0,
$$ 
we have
\bea
r_g(e_i, e_i) = \frac{s_g-\a}{n-1}.
\eea
More generally, we can show that
\bea
r_g(e_i, e_j) = \frac{s_g-\a}{n-1} \delta_{ij},\quad r_g(N, N) = \a, \quad r_g(N, e_i) = 0.\label{eqn3-25-2}
\eea
\qed}
\end{remark}

Next, we will show that every geometric quantity, including the function $\a$, is constant along each
level set of $f$.
Let $c$ be a regular value of $f$, so that $\Sigma = f^{-1}(c)$ is a hypersurface of $M$.
Let $X$ be a vector tangent to $\Sigma = f^{-1}(c)$. Then,
$$
X(|\n f|^2) = 2 Ddf(X, \n f) = - 2 r_g(X, \n f) = -\langle \n s_g, X\rangle =  0,
$$
since $\n f$ and $\n s_g$ are parallel. Therefore,  $|\n f|$ is constant on $f^{-1}(c)$. 

Let $\{e_1, \cdots, e_{n-1}, N\}$ be a local frame. Then, by Lemma~\ref{lem2015-7-14-10} together with
the Ricci soliton equation (\ref{eqn1-1-3}), we have
\bea
Ddf(N, e_i) = 0
\eea
for all $i = 1, \cdots, n-1$. Thus,
\bea
D_NN &=& \langle D_NN, N\rangle N + \sum_{i=1}^{n-1} \langle D_NN, e_i \rangle e_i\\
&=&
 \sum_{i=1}^{n-1} N\left(\frac{1}{|df|}\right) \langle \n f, e_i \rangle e_i +
  \sum_{i=1}^{n-1}  \frac{1}{|df|} \langle D_N df, e_i\rangle e_i = 0.
\eea
To show that the function $\a = r_g(N, N)$ is constant along each level hypersurface of $f$, recall that (\ref{eqn2013-9-13-7})
$$
D_X r_g(Y, \n f) = D_Y r_g(X, \n f)
$$ 
for any vectors $X$ and $Y$. In particular,
\be
D_Xr_g(\n f, \n f) = D_{\n f}r_g(X, \n f) \label{eqn2013-9-15-11}
\ee
for any vector $X$. Let $X$ be a vector with $X \perp \n f$.
Since $D_N N = 0$,  it follows from  (\ref{eqn2013-9-15-11}) that
\bea
\frac{1}{2} X(\a) &=& X\left(\frac{1}{|\n f|^2} r_g(\n f, \n f)\right)\\
&=&
\frac{1}{|\n f|^2} X(r_g(\n f, \n f)) = \frac{1}{|\n f|^2} \left[D_Xr_g(\n f, \n f) + 2r_g(D_Xdf, \n f)\right]\\
&=&
\frac{1}{|\n f|^2} \left[D_{\n f}r_g(X, \n f) + 2r_g(D_Xdf, \n f)\right]\\
&=&
\frac{1}{|\n f|^2} \left[\n f (r_g(X, \n f)) - r_g(D_{\n f}X, \n f) - r_g(X, D_{\n f}df)\right]\\
&=& 0.
\eea
Thus, $\a = r_g(N, N)$ is also constant   on each level set $f^{-1}(c)$. 
Consequently, we have the following result.

\begin{lem}\label{lem2013-11-22-1}
Let $(M^n, g, f)$ be a gradient  Ricci soliton with weakly harmonic Weyl tensor.
Then, $f, \, s_g, \,  \a, \,  |\n f|^2$,
and $2r_g(\n f, \n f) = \langle \n s_g, \n f\rangle$
are all constant along each level hypersurface given by $f$. 
\end{lem}

As mentioned above, since $|\a| \le |r_g|$, $\a= r_g(N, N)$ 
can be extended as a $C^0$ function on the whole of $M$.
The following lemma, in particular, shows that the function $\a$ is equal to $\rho$ on the set ${\rm Crit}(f)$.

\begin{lem}\label{lem2-6-5}
Let $(M^n, g, f)$ be a gradient  Ricci soliton with weakly harmonic Weyl tensor.
Then, on the set  ${\rm Crit}(f)$, we have that $\a =  \r.$
\end{lem}
\begin{proof}
From the Ricci soliton equation (\ref{eqn1-1-3}),  we have
\bea
N(|\n f|^2) = 2Ddf(N, \n f) = 2|\n f|(\rho - \a).\label{eqn2013-10-3-4-1}
\eea
Thus, 
\bea
N(|\n f|) = \rho - \a.\label{eqn2013-10-3-2-1}
\eea
Since $\a$ can be extended as a $C^0$ function on the whole of $M$, $|\n f|$ can be  
considered as a $C^1$ function
on the whole of $M$, including the critical set ${\rm Crit}(f)$ of $f$. Since $|\n f|$ attains its minimum on the set
${\rm Crit}(f)$, we have $N(|\n f|) = \rho - \a = 0$ on the set ${\rm Crit}(f)$ .
\end{proof}

\begin{lem}\label{lem2015-6-30-2}
Let $(M^n, g, f)$ be a complete gradient shrinking Ricci soliton.
If $f$ attains its (local) maximum at $x_0 \in M$, then
the scalar curvature $s_g$ also attains its (local) maximum at the point $x_0$.
\end{lem}
\begin{proof}
Recall that for a gradient shrinking Ricci soliton $(M,g, f)$,
\bea
s_g + |\n f|^2 - 2\rho f = C{\rm (constant)}.\label{eqn2013-12-15-1-1}
\eea
If  $f$ attains its local maximum at $x_0 \in M$, then ${\n f}(x_0) = 0$, and so at any point $x$ near $x_0$
 we have
$$
s_g(x_0) =  C + 2 \r f(x_0) \ge C+ 2 \r f(x) \ge C+ 2 \r f(x) - |\n f|(x) = s_g(x).
$$
\end{proof}

\begin{prop}\label{lem2015-7-14-1}
Let $(M^n, g, f)$ be a complete gradient shrinking Ricci soliton, 
and assume that $\delta \mathcal W(\cdot, \cdot, \nabla f) = 0$.
Then,
\bea
\Delta \a = N(\a)|df| + 2\r \a - \frac{2\a (s-\a)}{n-1}.\label{eqn2015-7-14-2}
\eea
\end{prop}
\begin{proof}
Let $\{e_1, \cdots, e_{n-1}, N = \frac{\n f}{|\n f|}\}$ be a local frame, and
let $R_{ij} = r_g(e_i, e_j)$ so that $\a:= R_{nn} = r_g(N, N)$. It has been shown  in \cite{e-n-m} that
$$
\Delta \a = \langle \n \a, \n f\rangle + 2\r \a - 2 \sum_{k, l}R_{nknl} R^{kl}.
$$
Finally, from Lemma~\ref{lem2015-7-14-10}  it is easy to see that
$$
 \sum_{k, l}R_{nknl} R^{kl} = \frac{\a}{n-1}\sum_{k=1}^{n-1} R_{kk}
 = \frac{\a}{n-1}(s_g - \a).
 $$
\end{proof}

\section{Proof of Main Theorems}

In this section, we shall prove  our main results.

\begin{thm}
 Let $(M^n, g, f)$ be a
compact gradient shrinking Ricci soliton with weakly harmonic Weyl tensor.  Then, $(M, g)$ is Einstein.
\end{thm}
\begin{proof}
 Let
$$
\max_{x\in M} f(x) = f(x_0).
$$
By Lemma~\ref{lem2015-6-30-2},  the scalar curvature $s_g$
also attains its maximum at $x_0$, i.e.,
$$
s_g(x_0) = \max_M s_g.
$$
Thus, if  $\Delta f(x_0) = 0$, 
then $s_g \le n \r$, because $\Delta f = n\r - s_g$. 
This shows that $f$ is a subharmonic function, and so it
must be constant. Hence,  $(M, g)$ is Einstein.

We claim that 
$$
\Delta f(x_0) = 0.
$$
From $\n s_g = 2\a \n f$, we have
\be
\frac{1}{2} \Delta s_g =  \langle \n \a, \n f\rangle + \a \Delta f
=  \langle \n \a, \n f\rangle + \a  (n\rho - s_g).\label{eqn2015-12-31-1}
\ee
The equation (\ref{eqn1-5-2}) can be rewritten as
\be
\frac{1}{2} \Delta s_g &=& \a |\n f|^2  + s_g\left(\rho - \frac{s_g}{n}\right) - \left|r_g - \frac{s_g}{n}g\right|^2. \label{eqn2015-12-31-2}
\ee
Therefore,
\be
\a|\n f|^2 = n \left(\a - \frac{s_g}{n}\right)\left(\rho - \frac{s_g}{n}\right) + 
 \left|r_g - \frac{s_g}{n}g\right|^2 + \langle \n \a, \n f\rangle.\label{eqn2015-12-31-3}
\ee
Thus, at the maximum point $x_0$ of $f$, we have from by  Lemma~\ref{lem2-6-5} that
$$
r_g = \frac{s_g}{n}g
$$
and 
$$
\a = \rho = \frac{s_g}{n}.
$$
Thus, 
$$
r_g = \frac{s_g}{n}g = \rho g
$$
at the point $x_0$, and so $Ddf_{x_0} = 0$ by the Ricci  soliton equation (\ref{eqn1-1-3}). 
Hence, $\Delta f(x_0)= 0$ and so $s_g$ is constant.

\end{proof}

Let  $(M^n, g, f)$ be a  noncompact gradient shrinking Ricci soliton.
 To prove Theorem B, we need the following Liouville
property for $f$-Laplacian functions, which are shown by  Petersen and
Wylie(\cite{p-w}).  The $f$-Laplacian of a function $u$ on $M$ is defined by
$$
\Delta_f u = \Delta u - \langle \n f, \n u\rangle.
$$

\begin{lem}[\cite{p-w}]\label{lem2015-5-30-1}
Any nonnegative function $u$ with $\Delta_f u \ge 0$ that satisfies
\be \lim_{r\to \infty}\left(\frac{1}{r^2}\int_{B(p, r)} u^k e^{-f}\,
dv_g \right) = 0\label{eqn2015-5-30-3} \ee for some $k >1$ is
constant.
\end{lem}

Using Lemma~\ref{lem2015-5-30-1}, we can prove the following
corollary.

\begin{cor}[\cite{p-w}]\label{lem2015-5-30-2}
Let $(M, g, f)$ be a complete gradient Ricci soliton. For a function
$u : M \to {\Bbb R}$, let
$$
\Omega_{u, C} := \{x\in M\,:\, u(x) \ge C\}.
$$
If $\Delta_f u \ge 0$ on $\Omega$ and satisfies
(\ref{eqn2015-5-30-3}),
 then $u$ is either constant or $u \le C$.
\end{cor}

\begin{thm}\label{main02}
Let $(M^n, g, f)$ be a complete noncompact gradient shrinking Ricci soliton with weakly harmonic Weyl tensor.
Then, $(M, g)$ is rigid.
\end{thm}
\begin{proof}
Let
$$
\Omega := \{x\in M\,:\, \r \Delta f - |Ddf|^2 <0\}.
$$
First, assume that  $\Omega = \emptyset$.
Let $u:= 2\r f - |df|^2$. Using the identity (\ref{eqn2015-6-14-2-1}), we can easily compute the following: 
\be
\Delta_f u = \Delta_f (2\r f-|df|^2) = 2\left(\r \Delta f - |Ddf|^2\right).\label{eqn3-25-4}
\ee
Since $\Omega = \emptyset$, we have
$$
\Delta_f u \ge 0.
$$
We can also easily see that the function $u= 2 \r f - |df|^2$
  satisfies (\ref{eqn2015-5-30-3}) for some $k>1$ (in fact, for any $k>1$). Thus, by
Lemma~\ref{lem2015-5-30-1}, $u = 2\r f - |df|^2$ is constant. Since
it is well known that $2\r f- |df|^2 = s_g +C$, it follows that the scalar curvature
$s_g$ must be constant. Therefore, $(M, g, f)$ is rigid by Theorem~\ref{thm2013-12-9-1}.

Now, assume that  $\Omega \ne \emptyset$.
In this case, we can compute
$$
\Delta_f (e^{-s_g}) = 2e^{-s_g}\left(|Ddf|^2 - \r \Delta f + 2\a^2
|df|^2\right).
$$
Therefore, $\Delta_f (e^{-s_g}) \ge 0$ on the set $\Omega$,
 and the function $e^{-s_g}$ satisfies (\ref{eqn2015-5-30-3})  obviously. Note that on
the set $\Omega$, it trivially holds that $e^{-s_g} > 0$. 
By Lemma~\ref{lem2015-5-30-2}, $e^{-s_g}$ is either constant or $e^{-s_g}
\le 0$. Because the latter condition is impossible, $e^{-s_g}$ must be constant,
and so is $s_g$. Hence, $(M, g,f)$ is also rigid in this case.

\end{proof}

\begin{remark}
{\rm 
When a gradient shrinking Ricci soliton $(M, g, f)$ does not satisfy the weakly harmonic Weyl 
condition (\ref{eqn3-25-6}),  we cannot be sure whether  (\ref{eqn3-25-4}) holds or not. In fact, there are gradient shrinking Ricci solitons whose scalar curvatures are not constant.
}
\end{remark}

\end{document}